\documentclass[12pt]{amsart}

\usepackage{amsmath, amssymb, bbm, amsbsy}
\DeclareMathOperator{\trace}{trace}
\usepackage{enumitem}
\newtheorem{theorem}{Theorem}[section]
\newtheorem{lemma}[theorem]{Lemma}

\newtheorem{cor}[theorem]{Corollary}

\theoremstyle{definition}
\newtheorem{defi}[theorem]{Definition}
\newtheorem{example}[theorem]{Example}

\theoremstyle{remark}
\newtheorem{remark}[theorem]{Remark}

\numberwithin{equation}{section}

\newcommand{\rr}{{\mathbb R}}
\newcommand{\rd}{{\mathbb R^d}}
\newcommand{\rmm}{{\mathbb{R}^m}}

\newcommand{\nat}{{\mathbb N}}
\newcommand{\ganz}{{\mathbb Z}}
\newcommand{\Exp}{{\mathbb E}}

\newcommand{\Gr}{\operatorname{Gr}}
\newcommand{\dimh}{\dim_{\mathcal{H}}}

\newcommand{\cardim}{\operatorname{cardim}}
\newcommand{\llambda}{{\pmb{\lambda}}}

\allowdisplaybreaks

\begin{document}
\sloppy
\title[Carrying dimension of self-affine occupation measures]{On the carrying dimension of occupation measures for self-affine random fields} 

\author{P. Kern}
\address{Peter Kern, Mathematical Institute, Heinrich-Heine-University D\"usseldorf, Universit\"atsstr. 1, D-40225 D\"usseldorf, Germany}
\email{kern\@@{}hhu.de} 

\author{E. S\"onmez}
\address{Ercan S\"onmez, Mathematical Institute, Heinrich-Heine-University D\"usseldorf, Universit\"atsstr. 1, D-40225 D\"usseldorf, Germany}
\email{ercan.soenmez\@@{}hhu.de} 

\date{\today}

\begin{abstract}
Hausdorff dimension results are a classical topic in the study of path properties of random fields. This article presents an alternative approach to Hausdorff dimension results for the sample functions of a large class of self-affine random fields. We present a close relationship between the carrying dimension of the corresponding self-affine random occupation measure introduced by U.\ Z\"ahle and the Hausdorff dimension of the graph of self-affine fields. In the case of exponential scaling operators, the dimension formula can be explicitly calculated by means of the singular value function. This also enables to get a lower bound for the Hausdorff dimension of the range of general self-affine random fields under mild regularity assumptions. 
\end{abstract}

\thanks{This work has been partially supported by Deutsche Forschungsgemeinschaft (DFG) under grant KE1741/6-1}
\maketitle

\baselineskip=18pt

\section{Introduction}

Let $U \in \mathbb{R}^{d \times d}$ and $V \in \mathbb{R}^{m \times m}$ be contracting,  non-singular matrices, i.e.\ $|\rho|\in(0,1)$ for any eigenvalue $\rho$ of $U$, respectively $V$. According to \cite[Definition 4.1]{Z3} a random field $X=\{X(t)\}_{t\in\rd}$ on $\rmm$ defined on a probability space $(\Omega, \mathcal{A}, P)$ is called $(U,V)$-{\it self-affine} if the following four conditions hold:
\begin{itemize}
\item[(i)] The field obeys the scaling relation
\begin{equation*}
	\{X( Ut)\}_{t \in\rd}\stackrel{\rm fd}{=} \{ V X(t)\}_{t \in\rd},
\end{equation*}
where ``$\stackrel{\rm fd}{=}$'' denotes equality of all finite-dimensional marginal distributions.
\item[(ii)] $X$ has {\it stationary increments}, i.e.\ $X(t)-X(s)\stackrel{\rm d}{=}X(t-s)$ for all $s,t\in\rd$, where ``$\stackrel{\rm d}{=}$'' denotes equality in distribution.
\item[(iii)] $X$ is {\it proper}, i.e.\ $X(t)$ is not supported on any lower dimensional hyperplane of $\rmm$ for all $t \neq 0$.
\item[(iv)] The mapping $X:\rd\times \Omega\to\rmm$ is ($\mathcal B(\rd)\otimes\mathcal A) - \mathcal B(\rmm)$-measurable with respect to the Borel-$\sigma$-algebras of $\rd$ and $\rmm$. 
\end{itemize}
Note that (i) implies $X(0)=0$ almost surely and that (iv) is fulfilled if $X$ has continuous sample functions (or right-continuous sample paths in case $d=1$).
Inductively, from (i) we get $\{X( U^nt)\}_{t \in\rd}\stackrel{\rm f.d.}{=} \{ V^n X(t)\}_{t \in\rd}$ for every $n\in\ganz$ and thus self-affinity weakens the assumption of self-similarity \cite{Lamperti,LahaRo,EmbrMaej,DidMeePip} to a discrete scaling property, which is also called semi-selfsimilarity \cite{MaeSat} in the context of stochastic processes, where $d=1$ and usually we have the restriction $t\geq0$.

Over the last decades there has been increasing attention in such random fields in theory as well as in applications. Posssible applications can be found in such diverse fields as engineering, finance, physics, hydrology, image processing or network analysis; e.g., see \cite{Benson,Bianchi,Chiles,Davies,DP1,DP2,Harba,Levy,Ponson,Roux,WillPaxTaqq} and the literature cited therein. 
Particularly, in the study of sample path behavior, it is of considerable interest to determine fractal dimensions such as Hausdorff dimension of random sets depending on the sample paths of a self-affine random field. E.g., we refer to \cite{Fal,Mattila} for a comprehensive introduction to fractal geometry and the notion of Hausdorff dimension.

The main objective in this paper is the occupation measure $\tau_X$ of a self-affine random field $X$, measuring the size in $\rd$ the graph of $X$ spends in a Borel set of $\rd\times\rmm$ with respect to Lebesgue measure. In a series of papers \cite{Z1,Z2,Z3} U.\ Z\"ahle investigated this object in detail, which serves as a starting point of  our considerations. In particular, Z\"ahle \cite{Z3} showed that the occupation measure $\tau_X$ of a self-affine random field $X$ is Palm distributed and itself a self-affine random measure, see Section 2 for details. This allows to study Hausdorff dimension results through the notion of carrying dimension of $\tau_X$ introduced in \cite{Z1}. Heuristically, the carrying dimension of a Borel measure is the minimal Hausdorff dimension for Borel sets assigning positive measure; see Definition \ref{22cardim} below for the precise mathematical description. From the definition it is  obvious that the Hausdorff dimension of the graph of a self-affine random field is bounded from below by the carrying dimension of the occupation measure $\tau_X$ if the latter exists. Our main aim is to show that under a natural condition the carrying dimension of $\tau_X$ exists and almost surely coincides with the Hausdorff dimension of the graph of $X$. This gives the perspective to calculate the Hausdorff dimension of the graph of $X$ by means of the carrying dimension of $\tau_X$ and vice versa. Under the mild additional assumption of boundedly continuous intensity, see Definition \ref{24bci} below, U.\ Z\"ahle \cite{Z3} further showed that the carrying dimension of $\tau_X$ can be calculated by means of the singular value function. This well known method is suitable in fractal geometry to derive the Hausdorff dimension of self-affine sets arising from iterated function systems \cite{Fal1988,Fal1992,Fal2013} and strengthens our approach. It enables us to also derive a lower bound for the Hausdorff dimension of the range of self-affine random fields under the boundedly continuous intensity assumption.

Much effort has been made in the last decades in order to calculate the Hausdorff dimension of the range and the graph of the paths arising from several special classes of self-affine random fields. Classically, the Hausdorff dimension is determined by calculating an upper and a lower bound separately. This approach requires an a priori educated guess on the true value of the Hausdorff dimension. A typical method in the calculation of an upper bound is to find an efficient covering of the graph for example by using sample path properties such as H\"older continuity or independent increments, whereas the calculation of a lower bound is usually related to potential theoretic methods. A further aim of the present paper is to provide candidates for the Hausdorff dimension of the graph and the range of self-affine random fields in case the contracting non-singular operators $U$ and $V$ are given by exponential matrices, that is $U= c^E$ and $V=c^D$, where $0<c<1$ and $E \in \mathbb{R}^{d \times d}$, $D \in \mathbb{R}^{m \times m}$ are matrices with positive real parts of their eigenvalues.  
In many situations the appearance of exponential scaling matrices is quite natural in the context of self-similar or self-affine random fields and processes; see \cite{HudsonMason, MS, LiXiao, DidMeePip}. Under the condition of boundedly continuous intensity the above candidates always serve as lower bounds for the Hausdorff dimension of the graph and the range of self-affine random fields. Known methods to derive corresponding upper bounds heavily depend on further properties of the field such as H\"older continuity or independent increments and should be derived case by case elsewhere. 
In our approach we particularly elucidate the intuition that the Hausdorff dimension of the graph and the range over sample paths of self-affine random fields should only depend on the real parts of the eigenvalues of $E$ and $D$ as well as their multiplicity.

The rest of this article is structured as follows. Section 2 basically serves as an introduction to self-affine random measures as given in \cite{Kal,Z1,Z2,Z3} which will be applied in order to establish the main results of this paper. Here, we adopt some notation and repeat fundamental notions and results from \cite{Z1,Z2,Z3} concerning self-affine random measures, Palm distributions, the carrying dimension and the boundedly continuous intensity condition. Section 3 is the core part of this article, where we formulate and prove the above mentioned main results. Finally, in Section 4 we show that our results can be applied to large classes of self-affine random fields, namely to operator-self-similar stable random fields introduced by Li and Xiao \cite{LiXiao}, and to operator semistable L\'evy processes. For these particular classes of self-affine random fields, our candidates derived by means of the singular value function in Section 3 are in fact the true values for the Hausdorff dimension of the graph and the range as recently shown in \cite{Soenmez1,Soenmez2,Kern,Wedrich}. Furthermore, our results may be useful to derive Hausdorff dimension results for classes of random processes and fields, for which this still remains an open question, e.g.\ for multiparameter operator semistable L\'evy processes or certain semi-selfsimilar Markov processes.

\section{Preliminaries}

In this section we recall some basic facts on random measures and Palm distributions which will be needed for our approach. We further introduce the main objectives, the occupation measure of a self-affine random field and its carrying dimension.

\subsection{Random measures, Palm distributions and occupation measures}

Let $\mathcal{M}_{\mathbb{R}^n}$ be the set of all locally finite measures on $\mathbb{R}^n$ equipped with the corresponding $\sigma$-algebra $\mathcal{B} (\mathcal{M}_{\mathbb{R}^n})$ {generated by the mappings $\mathcal{M}_{\mathbb{R}^n}\ni\mu\mapsto\mu(B)$ for all bounded sets $B\subset\mathbb R^n$}; e.g., see \cite{Kal} for details. A random variable {$\xi:\Omega\to\mathcal{M}_{\mathbb{R}^n}$} is called a {\it random measure} on $\mathbb{R}^n$. For any random measure $\xi$ denote by $P_\xi$ its distribution. Note that $P_\xi$ is a probability measure on $( \mathcal{M}_{\mathbb{R}^n} , \mathcal{B} (\mathcal{M}_{\mathbb{R}^n})) $. Furthermore, it is clear that the mapping
$$\mathbb{E} [ \xi ]  :  {\mathcal{B}(\mathbb{R}^n)} \to [0, \infty ],\quad A \mapsto \mathbb{E} [ \xi (A) ]=\int_{\mathcal{M}_{\mathbb{R}^n}}\mu(A)\,dP_{\xi}(\mu)$$
is a (deterministic) Borel measure called the {\it intensity measure} of $\xi$. 

For any measure $\mu \in \mathcal{M}_{\mathbb{R}^n}$ and $z \in \mathbb{R}^n$ let $T_z \mu$ be the translation measure given by $ T_z \mu (B) = \mu (B-z)$
for all $B \in \mathcal{B} (\mathbb{R}^n)$. We say that a random measure on $\mathbb{R}^n$ is {\it stationary} if
$ P_\xi = P_{T_z \xi}$
for any $z \in \mathbb{R}^n$.
Let {$x>0$} and $W \in \mathbb{R}^{n \times n}$ be a non-singular operator. A random measure $\xi$ is called $(W,x)$-{\it self-affine} if
$$ \xi (A) \stackrel{\rm d}{=} x \cdot \xi (W^{-1} A)\quad\text{ for any }A  \in \mathcal{B} (\mathbb{R}^n).$$

We now turn to the definition of Palm distributions. A $\sigma$-finite measure on $\big( \mathcal{M}_{\mathbb{R}^n} , \mathcal{B} (\mathcal{M}_{\mathbb{R}^n})\big) $ shall be called {\it quasi-distribution}. Again, a translation invariant quasi-distribution $Q$ satisfying
$$ dQ(\mu) = dQ(T_z \mu)\quad\text{ for every }z \in \mathbb{R}^n$$
is called {\it stationary}. For a stationary quasi-distribution it is easy to see that the Borel measure $A\mapsto\int_{\mathcal{M}_{\mathbb{R}^n}} \mu (A)\,dQ(\mu)$ is translation invariant and thus there exists a constant $c_Q \in [0, \infty]$ such that 
$$ \int_{\mathcal{M}_{\mathbb{R}^n}} \mu (A)\,dQ(\mu) = c_Q \cdot \llambda_n (A)\quad\text{ for any }A \in \mathcal{B} (\mathbb{R}^n), $$
where $\llambda_n$ denotes the Lebesgue measure on $\mathbb{R}^n$. Note that for a stationary random measure $\xi$ this implies that there is a constant $c_\xi \in [0, \infty]$ such that the intensity measure satisfies
$\mathbb{E} [\xi] = c_\xi \cdot \llambda_n$. Throughout this paper, we will refer to $c_\xi$ and $c_Q$ as the {\it intensity constants} of $\xi$, respectively $Q$. For a stationary quasi-distribution $Q$ the measure $Q^0$ defined by
\begin{equation} \label{21palm}
Q^0 (G) = \frac{1}{\llambda_n(A)} \int_{\mathcal{M}_{\mathbb{R}^n}} \int_A \mathbbm{1}_G(T_{-z} \mu)\, d \mu (z) \,dQ(\mu)\quad\text{ for all } G \in \mathcal{B} (\mathcal{M}_{\mathbb{R}^n}) ,
\end{equation} 
is independent of the choice of $A  \in \mathcal{B} (\mathbb{R}^n)$ as long as $0 < \llambda_n(A) < \infty$ and it is called the {\it Palm measure} of $Q$; see \cite[page 85]{Z1}. Note that $Q^0\ll Q$ by \eqref{21palm}, i.e.\ any $Q$-nullset is also a $Q^0$-nullset. A random measure $\xi$ is called {\it Palm distributed} if there exists a stationary quasi-distribution $Q$ with Palm measure $Q^0$ such that
\begin{equation}\label{21palma}
P_\xi = c_Q^{-1} \cdot Q^0,
\end{equation}
where $c_Q=Q^0(\mathcal{M}_{\mathbb{R}^n}) \in (0, \infty)$ is the intensity constant of $Q$.

Let $f: \rd \to \rmm$ be a Borel-measurable function. Then the {\it occupation measure} of $f$ is a Borel measure $\tau_f$ on $\rd \times \rmm$ uniquely defined by
$$ \tau_f (A \times B) = \llambda_d \{ t \in A : f(t) \in B \} $$
for all $A  \in \mathcal{B} (\mathbb{R}^d)$, $B  \in \mathcal{B} (\mathbb{R}^m)$. Note that $\tau_f$ is concentrated on the graph of $f$ and, to be more precise, we may also say that $\tau_f$ is the occupation measure of the graph. For any  $(U,V)$-self-affine random field $\{X(t)\}_{t\in\rd}$ on $\rmm$ we get by the transformation rule
\begin{align*}
\tau_X(A\times B) & =\int_A\mathbbm{1}_B(X(t))\,d\llambda_d(t)=\det U\cdot\int_{U^{-1}(A)}\mathbbm{1}_B(X(Ut))\,d\llambda_d(t)\\
& \stackrel{\rm d}{=}\det U\cdot\int_{U^{-1}(A)}\mathbbm{1}_B(VX(t))\,d\llambda_d(t)\\
& =\det U\cdot\int_{U^{-1}(A)}\mathbbm{1}_{V^{-1}(B)}(X(t))\,d\llambda_d(t)\\
& =\det U\cdot\tau_X(U^{-1}(A)\times V^{-1}(B))=\det U\cdot\tau_X((U \oplus V)^{-1}(A\times B)),
\end{align*}
where $U \oplus V{\in\rr^{(d+m)\times(d+m)}}$ denotes the block-diagonal matrix. Hence the occupation measure $\tau_X$ defines a $(U \oplus V, \det U)$-self-affine random measure. Moreover, it is Palm distributed by the following Lemma due to U.\ Z\"ahle \cite{Z3}.

\begin{lemma}{\cite[Proposition 5.2]{Z3}} \label{21om}
Let $\{X(t)\}_{t\in\rd}$ be a $(U,V)$-self-affine random field on $\rmm$ as defined in Section 1. Denote by $\tau_X$ its occupation measure on $\rd \times \rmm$. Then $\tau_X$ is a Palm distributed and $(U \oplus V, \det U)$-self-affine random measure.
\end{lemma}

\subsection{Carrying dimension}

We now introduce the notion of carrying dimension defined in \cite{Z1} and recall a result from \cite{Z3} on how the carrying dimension of the occupation measure of self-affine random fields, under certain regularity assumptions, can be explicitely calculated.

\begin{defi} \label{22cardim}
Let $\mu \in \mathcal{M}_{\mathbb{R}^n}$ be a Borel measure on $\mathbb{R}^n$. We say that $\mu$ has  {\it carrying dimension} $\mathfrak d\in[0,n]$, in symbols $\mathfrak d=\cardim \mu$, if the following two conditions are satisfied.
\begin{enumerate}
\item[(i)] $\mu (A)>0$ implies $\dimh A \geq{\mathfrak d}$ for any $A  \in \mathcal{B} (\mathbb{R}^n)$.
\item [(ii)] {There} exists a set $B  \in \mathcal{B} (\mathbb{R}^n)$ with $\mu (\mathbb{R}^n \setminus B ) = 0$ and $\dimh B \leq{\mathfrak d}$.
\end{enumerate}
\end{defi}

This definition is closely related to the lower and upper Hausdorff dimension of the Borel measure $\mu$ given by
\begin{align*}
\dim_\ast\mu & =\inf\{\dim_{\mathcal H}A:A\in\mathcal B(\rr^n),\,\mu(A)>0\},\\
\dim^\ast\mu & =\inf\{\dim_{\mathcal H}B:B\in\mathcal B(\rr^n),\,\mu(\rr^n\setminus B)=0\},
\end{align*}
and discussed in \cite{Cut,HuTay,Edg}. Obviously, the carrying dimension of $\mu$ exists if and only if $\dim_\ast\mu=\dim^\ast\mu$ and in this case these values are all equal.

The following Lemma of U.\ Z\"ahle \cite{Z2} is useful to derive a lower bound of the carrying dimension. Its proof can be found in \cite[Theorem 1.4]{Zalt}. 

\begin{lemma}{ \cite[Lemma 2.1]{Z2}} \label{32lowerbound}
Let $\mu \in \mathcal{M}_{\mathbb{R}^n}$, $\gamma \geq 0$ and $B  \in \mathcal{B} (\mathbb{R}^n)$. Suppose that
$$ \int_{\{\| z-x \| <1\} } \| z-x \|^{-\gamma} \,\mu (dz) < \infty $$
for $\mu$-almost all $x \in B$. Then $\mu ( B' \cap B) > 0$ implies $\dimh B' \geq \gamma$ for any $B'  \in \mathcal{B} (\mathbb{R}^n)$ and, consequently, $\cardim \mu \geq \gamma$.
\end{lemma}

For an explicit calculation of the carrying dimension of the occupation measure $\tau_X$ of a $(U,V)$- self-affine random field $\{X(t)\}_{t\in\rd}$ the following condition, called the {\it boundedly continuous intensity} (b.c.i.) condition in \cite{Z3}, is crucial and sufficient. 

\begin{defi}\label{24bci}
Let $\xi$ be a $(W,x)$-self-affine random measure. Then $\xi$ is said to satisfy the b.c.i.\ condition (with respect to $W$) if there exists a constant $0<C<\infty$, not depending on $W$, such that
$$\mathbb{E} [ \xi](A)\leq C \cdot \llambda_{d+m}(A)\quad\text{ for any Borel set } A\subset [-1,1]^{d+m} \setminus W([-1,1]^{d+m}).$$
\end{defi}

An easy sufficient condition for the occupation measure to fulfill the b.c.i.\ condition is the following.

\begin{lemma}\label{specialocm}
Suppose that for any $t\in\rd\setminus\{0\}$ the distribution of $X(t)$ has a density $x\mapsto p_t(x)$ with respect to $\llambda_m$, then for $z=(t,x)\in\rd\times\rmm$ we have
$$d\mathbb{E} [ \tau_X](z)=p_t(x)\,d\llambda_d(t)\,d\llambda_m(x).$$
Furthermore, if there exists a constant $0<C<\infty$ such that $p_t(x)\leq C$ for any $(t,x)\in[-1,1]^{d+m} \setminus W([-1,1]^{d+m})$
it follows immediately that the b.c.i.\ condition is fulfilled.
\end{lemma}

\begin{proof}
For any $A\in\mathcal B(\rd)$, $B\in\mathcal B(\rmm)$ by Tonelli's theorem we get
\begin{equation}\label{intmeasdens}\begin{split}
\mathbb{E} [ \tau_X](A\times B) & =\int_{\Omega}\int_A\mathbbm 1_B(X(t))\,d\llambda_d(t)\,dP\\
&=\int_AP\{X(t)\in B\}\,d\llambda_d(t)=\int_{A\times B}\,dP_{X(t)}(x)\,d\llambda_d(t)
\end{split}\end{equation}
and in case $dP_{X(t)}(x)=p_t(x)\,d\llambda_m(x)$ the assertion follows.
\end{proof}

Z\"ahle \cite{Z3} showed that under the b.c.i.\ condition there is a close relation between the carrying dimenison of occupation measures and the singular value function, which is frequently used as a tool in the study of the Hausdorff dimension of self-affine fractals; {e.g., see \cite{Fal1988, Fal1992, Fal2013}}. Following \cite{Fal1988}, let us briefly introduce the singular value function $\phi_W$ of a contracting, non-singular matrix $W \in \mathbb{R}^{n \times n}$. Let $1 > \alpha_1 \geq \alpha_2 \geq \ldots \geq \alpha_n >0$ denote the singular values of $W$, i.e.\ the positive square roots of the eigenvalues of $W^\top W$, where $W^\top$ denotes the transpose of $W$. Then the singular value function $\phi_W:(0,n]\to(0,\infty)$ of $W$ is given by
\begin{equation}\label{svfct}
\phi_W (s) = \alpha_1\cdot \alpha_2 \cdot\ldots \cdot\alpha_{m-1} \cdot\alpha_{m}^{s-m+1},
\end{equation}
where $m$ is the unique integer such that $m-1< s\leq m$. 

\begin{lemma}{\cite[Proposition 4.1]{Fal1988}} \label{23svf}
Let $W \in \mathbb{R}^{n \times n}$ be a contracting and non-singular matrix, $0<x<1$ and $\phi_W$ the singular value function of $W$. {Then} there exists a unique number $s = s(W,x)>0$ given by $x^{-1}(\phi_{W^k}(s))^{\frac{1}{k}}\to1$ as $k\to\infty$. Moreover, it holds that
$$s = \inf \left\{ r\in(0,n] : \lim_{k\to\infty}x^{-k}\phi_{W^k}(r)=0\right\}=\sup\left\{ r\in(0,n] : \lim_{k\to\infty}x^{-k}\phi_{W^k}(r)=\infty\right\}$$
with the convention that $\inf\emptyset=n$.
\end{lemma}
The following result of U.\ Z\"ahle \cite{Z3} will be important for our approach and states that under the b.c.i.\ condition the carrying dimension of the occupation measure of any self-affine random field can be calculated in terms of the singular value function.

\begin{theorem}{\cite[Theorem 5.3]{Z3}} \label{24cardim}
Let {$X=\{X(t)\}_{t\in\rd}$} be a $(U,V)$-self-affine random field {on $\rmm$} and $\tau_X$ its occupation measure. If $\tau_X$ satisfies the b.c.i.\ condition with respect to $W = U \oplus V$ then with probability one
$$ \cardim \tau_X = s (W, \det U) ,$$
where $s (W, \det U)$ is the unique number given {by} Lemma \ref{23svf}.
\end{theorem}


\section{Main results}

Throughout this section, let $X=\{X(t)\}_{t\in\rd}$ be a $(U,V)$-self-affine random field on $\rmm$ as introduced in Section 1 and denote by $\tau_X$ its occupation measure for the graph. Moreover, denote by
$$ \Gr X ([0,1]^d) = \big\{ \big( t, X(t) \big) : t \in [0,1]^d \big\} \subset \mathbb{R}^{d+m}$$
the graph of $X$ on the unit cube. We now show that, under a natural additional assumption, the carrying dimension of $\tau_X$ coincides with the Hausdorff dimension of the graph of $X$.

\begin{theorem} \label{31main}
Assume that 
\begin{equation}\label{3Frostmancondition}
\int_{[-1,1]^d}\mathbb{E} \big[ \big( \|t\| + \|X(t) \| \big) ^{-\gamma} \big] \,d\llambda_d(t) < \infty
\end{equation}
for any $\gamma < \dim_{\mathcal{H}} \Gr X ([0,1]^d)$.
Then with probability one the carrying dimension of $\tau_X$ exists and we have
\begin{equation*}
\cardim \tau_X = \dim_{\mathcal{H}} \Gr X ([0,1]^d).
\end{equation*}
\end{theorem}

\begin{remark}
Note that by the definition of the carrying dimension, the upper bound
$\cardim \tau_X \leq \dim_{\mathcal{H}} \Gr X ([0,1]^d)$ almost surely 
is immediate. Thus we only need to proof the lower bound. Further note that by stationarity of the increments, \eqref{3Frostmancondition} is equivalent to 
$$\int_{[0,1]^d \times [0,1]^d} \mathbb{E} \big[ \big( \|t-s\| + \|X(t) - X(s) \| \big) ^{-\gamma} \big]\,d\llambda_d(t)\,d\llambda_d(s) < \infty$$
and by Frostman's Theorem \cite{Fal, Kahane, Mattila} this implies that $\dim_{\mathcal{H}} \Gr X ([0,1]^d) \geq \gamma$ almost surely, which is the canonical tool to derive a lower bound of the Hausdorff dimension. Furthermore, the above integral is the expected value of the $\gamma$-energy of $\tau_X$, usually denoted $I_\gamma ( \tau_X)$. By Frostman's lemma \cite{Fal,Mattila} almost surely there exists a (random) probability distribution $\mu$ on $\Gr X ([0,1]^d)$ such that $I_\gamma (\mu) < \infty$ if $\gamma < \dim_{\mathcal{H}} \Gr X ([0,1]^d)$. However, in general one does not have the information that $\mu = \tau_X$, although $\mu = \tau_X$ is the canonical candidate for the derivation of a lower bound.
\end{remark}

\begin{proof}[Proof of Theorem \ref{31main}]
As remarked above, we only need to prove the lower bound
\begin{equation}\label{3proof1}
\cardim \tau_X \geq \dim_{\mathcal{H}} \Gr X ([0,1]^d) \quad \text{ almost surely.}
\end{equation}
By Lemma \ref{21om}, $\tau_X$ is Palm distributed and thus there exists a stationary quasi-distribution $Q$ with intensity constant $c_Q$ and Palm measure $Q^0$ given by \eqref{21palm} such that $P_{\tau_X} = c_Q^{-1}Q^0$ as in \eqref{21palma}. Moreover, to prove \eqref{3proof1}, by Lemma \ref{32lowerbound} it suffices to show that for any $\gamma<\dim_{\mathcal{H}} \Gr X ([0,1]^d)$ we have
\begin{equation}\label{suff1}
\int_{ \{\| z-x \| <1\} } \| z-x \|^{-\gamma} \,d\mu(z) < \infty 
\end{equation}
for $P_{\tau_X}$-almost all $\mu$ and $\mu$-almost all $x \in \Gr X ([0,1]^d)$.
If we can show that
\begin{equation}\label{3proof3}
\mathbb{E} \Big[ \int_{ \{\| z \|<1\} } \| z \|^{-\gamma} \,d\tau_X (z) \Big] < \infty
\end{equation}
then for $Q^0$-almost all $\mu$ we have
$$ \int_{ \{\| z \| <1\} } \| z \|^{-\gamma} \,d\mu(z) < \infty, $$
which by \cite[Lemma 3.3]{Z1} is equivalent to
$$ \int_{ \{\| z-x \| <1\} } \| z-x \|^{-\gamma} \,d\mu(z) < \infty $$
for $Q$-almost all $\mu$ and $\mu$-almost all $x \in \Gr X ([0,1]^d)$. Since $Q^0\ll Q$ by \eqref{21palm} and thus $P_{\tau_X}\ll Q$ by \eqref{21palma}, it follows that \eqref{suff1} holds for $P_{\tau_X}$-almost all $\mu$ and $\mu$-almost all $x \in \Gr X ([0,1]^d)$. Thus it suffices to show \eqref{3proof3}. By \eqref{intmeasdens} and our assumption \eqref{3Frostmancondition} we get for some constant $0<K<\infty$
\begin{align*}
\mathbb{E} \Big[ \int_{ \{\| z \| <1\} } \| z \|^{-\gamma} \,d\tau_X (z) \Big] & = \int_{ \{\| z \| <1\} } \| z \|^{-\gamma} \,d\mathbb E[\tau_X](z)\\
& =\int_{ \{\| (t,x) \| <1\} } \| (t,x) \|^{-\gamma} \,dP_{X(t)}(x)\,d\llambda_d(t)\\
& \leq K\int_{[-1,1]^d}\int_{\rmm}(\|t\|+\|x\|)^{-\gamma}\,dP_{X(t)}(x)\,d\llambda_d(t)\\
& =K \int_{[-1,1]^d}\mathbb{E} \big[ \big( \|t\| + \|X(t) \| \big) ^{-\gamma} \big] \,d\llambda_d(t)<\infty,
\end{align*}
which shows \eqref{3proof3} and concludes the proof.
\end{proof}

Combining Theorem \ref{31main} with Theorem \ref{24cardim} we immediately get the following result.

\begin{cor} \label{33computation}
Assume \eqref{3Frostmancondition} is fulfilled and $\tau_X$ satisfies the b.c.i.\ condition. Then with probability one
$$ \dim_{\mathcal{H}} \Gr X ([0,1]^d) = s (W, \det U) ,$$
where $W = U \oplus V$ and $s (W, \det U)$ is the unique number given by Lemma \ref{23svf}.
\end{cor}

In case the contracting, non-singular operators $U$ and $V$ are given by exponential matrices $U = c^E$ and $V=c^D$ for some $c\in(0,1)$ and some matrices $E\in\rr^{d\times d}$ and $D\in\rr^{m\times m}$ with positive real parts of their eigenvalues, we are able to calculate $s (U \oplus V, \det U)$ of Lemma \ref{23svf} explicitly in terms of the real parts of the eigenvalues of $E$ and $D$ as follows.

\begin{example}\label{sexpl}
Let $0<a_1\leq\cdots\leq a_d$ and $0<\lambda_1\leq\cdots\leq \lambda_m$ denote the real parts of the eigenvalues of $E$, respectively $D$. Write $0<\gamma_1\leq\cdots\leq \gamma_{d+m}$ for the union of all these quantities in a common order. Then we have
$$\det U=\det c^E=c^q,$$
where $q=\trace(E)=a_1+\cdots+a_d$. For the block-diagonal matrix $W=c^E\oplus c^D=c^{E\oplus D}$ we obtain that the positive square roots of the eigenvalues of the symmetric matrices $(W^n)^\top W^n$ asymptotically behave as $c^{n\gamma_j}$ for $j=1,\ldots,d+m$. More precisely, for any $\varepsilon>0$ the $j$-th smallest square root $\eta_n(j)$ of the eigenvalues of  $(W^n)^\top W^n$ fulfills
\begin{equation}\label{jsev}
c^{n(\gamma_j+\varepsilon)}\leq\eta_n(j)\leq c^{n(\gamma_j-\varepsilon)}
\end{equation}
for all $j=1,\ldots,d+m$ and $n\in\nat$ large enough; e.g.\ see section 2.2 in \cite{MS} for details. Now let $r\in\{1,\ldots, d+m\}$ be the unique integer such that
\begin{equation}\label{defr}
\sum_{j=1}^{r-1}\gamma_j<q\leq\sum_{j=1}^{r}\gamma_j
\end{equation}
then by \eqref{jsev} for any $r-1<s\leq r$ the singular value functon $\phi_{W^n}(s)$ in the above sense asymptotically behaves as 
$$c^{n\gamma_1}\cdots c^{n\gamma_{r-1}}c^{n\gamma_{r}(s-r+1)} $$
and a comparison with $(\det U)^n=c^{nq}$ together with \eqref{defr} readily shows that
\begin{equation}\label{specials}
s(c^{E\oplus D},c^q)=r-1+\frac1{\gamma_{r}}\,\Big(q-\sum_{j=1}^{r-1}\gamma_j\Big).
\end{equation}
Note that if $q=\sum_{j=1}^{r}\gamma_j$ then $s(W,\det U)=r$ which shows that $s(W,\det U)\geq d$. Note further that the right-hand side of \eqref{specials} is independent of $c\in(0,1)$ and only depends on the real parts of the eigenvalues of the scaling exponents $E$ and $D$.
\end{example}

\vspace*{2ex}
We now turn to the range $X([0,1]^d)=\{X(t):t\in[0,1]^d\}$ of the self-affine random field $X$ and its (random) occupation measure $\sigma_X:\Omega\to\mathcal M_{\rmm}$ defined as in \cite{Khosh} by
$$\sigma_X(B)=\int_{[0,1]^d}\mathbbm 1_B(X(t))\,d\llambda_d(t)=\lambda_d\{t\in[0,1]^d:X(t)\in B\}=\tau_X([0,1]^d\times B)$$
for $B\in\mathcal B(\rmm)$. Note that $\sigma_X$ is supported on the random set $X([0,1]^d)$ and has intensity measure
$$\Exp[\sigma_X](B)=\int_{[0,1]^d}P\{X(t)\in B\}\,d\llambda_d(t),$$
which in case $d=1$ is also called the expected sojourn time of $X$ on $[0,1]$ in the Borel set $B\in\mathcal B(\rmm)$. It is easy to see that for $B\in\mathcal B(\rmm)$ we have
$$\sigma_X(V^{-1}(B))\stackrel{\rm d}{=}(\det U)^{-1}\int_{U([0,1]^d)}\mathbbm 1_B(X(t))\,d\llambda_d(t),$$
which shows that $\sigma_X$ is not $(V,\det U)$-self-affine, since $[0,1]^d$ is not $U$-invariant. Hence an approach analogous to Theorem \ref{31main} for the graph, combining the Hausdorff dimension of the range with the carrying dimension of $\sigma_X$, fails; cf.\ also \cite{SX}. Nevertheless, if the carrying dimension of $\sigma_X$ exists, then obviously $\cardim\sigma_X\leq\cardim\tau_X$. Furthermore, it seems quite natural that the Hausdorff dimension of the range of the self-affine random field is connected to $s(V,\det U)$. In case $U=c^E$ and $V=c^D$ are given by exponential matrices for some $c\in(0,1)$ as above, we will now show that $s(V,\det U)=s(c^D,c^q)$ with $q=\trace(E)$ always serves as a lower bound for $\dim_{\mathcal H}X([0,1]^d)$ with probability one, provided that the b.c.i.\ condition for the occupation measure $\tau_X$ of the graph is fulfilled. We will first calculate $s(c^D,c^q)$ explicitly in terms of the real parts of the eigenvalues of $E$ and $D$.

\begin{example}\label{svalrange}
Let $0<a_1\leq\cdots\leq a_d$ denote the real parts of the eigenvalues of $E$ and let $0<\lambda_1<\cdots< \lambda_p$ be the distinct real parts of the eigenvalues of $D$ with multiplicities $m_1,\ldots, m_p$. Then $q=\trace(E)=\sum_{k=1}^da_k$ and we distinguish between the following two cases.

{\bf Case 1:} If for some $\ell\in\{1,\ldots,p\}$ we have
\begin{equation}\label{c1cond}
\sum_{i=1}^{\ell-1} \lambda_i m_i<q\leq\sum_{i=1}^{\ell} \lambda_i m_i
\end{equation}
then there exists $R\in\{1,\ldots, m_{\ell}\}$ such that
\begin{equation}\label{c1condref}
\sum_{i=1}^{\ell-1} \lambda_i m_i+\lambda_{\ell}(R-1)<q\leq\sum_{i=1}^{\ell-1} \lambda_i m_i+\lambda_{\ell}R
\end{equation}
so that for $W=c^D$ we have $r=\sum_{i=1}^{\ell-1} m_i+R$
in \eqref{defr}.  From \eqref{specials} it follows that
\begin{equation}\label{specialsc1}\begin{split}
s(c^D,c^q) & =\sum_{i=1}^{\ell-1} m_i+R-1+\frac1{\lambda_{\ell}}\Big(q-\sum_{i=1}^{\ell-1} \lambda_i m_i-\lambda_{\ell}(R-1)\Big)\\
& =\frac1{\lambda_{\ell}}\Big(q+\sum_{i=1}^{\ell} (\lambda_{\ell}-\lambda_i) m_i\Big).
\end{split}\end{equation}
Combining \eqref{specialsc1} with the second inequality in \eqref{c1cond} we see that
\begin{equation}\label{ubsv}
s(c^D,c^q) \leq\sum_{i=1}^{\ell} m_i\leq m.
\end{equation}

{\bf Case 2:} If $q>\sum_{i=1}^p\lambda_i m_i$ we choose $\varepsilon >0$ such that $q>\sum_{i=1}^p(\lambda_i+\varepsilon)m_i$ then by \eqref{jsev} we get as $n\to\infty$
$$c^{-nq}\phi_{c^{nD}}(m)\geq c^{-n(q-\sum_{i=1}^p(\lambda_i+\varepsilon)m_i)}\to\infty,$$
showing that $s(c^D,c^q)=m$.

Altogether, we have shown that irrespectively of $c\in(0,1)$ we have
\begin{equation}\label{srange}
s(c^D,c^q)=\begin{cases} \displaystyle\frac1{\lambda_{\ell}}\Big(q+\displaystyle\sum_{i=1}^\ell (\lambda_{\ell}-\lambda_i) m_i\Big) & \text{ , if }\displaystyle\sum_{i=1}^{\ell-1} \lambda_i m_i<q\leq\displaystyle\sum_{i=1}^{\ell} \lambda_i m_i\\
\,m & \text{ , else.}\end{cases}
\end{equation}
\end{example}

\begin{theorem}\label{lbrange}
Let $X=\{X(t)\}_{t\in\rd}$ be a $(c^E,c^D)$-self-affine random field on $\rmm$ for some $c\in(0,1)$ such that its occupation measure $\tau_X$ of the graph satisfies the b.c.i.\ condition. Then with probability one we have
$$ \dim_{\mathcal{H}}X ([0,1]^d)\geq s(c^D,c^q),$$
where $q=\trace(E)$ and $s(c^D,c^q)$ is given by \eqref{srange}.
\end{theorem}

\begin{proof}
By Frostman's Theorem \cite{Fal, Kahane, Mattila} it suffices to show that for any $\gamma<s(c^D,c^q)$ we have
\begin{equation}\label{suff}
\int_{[0,1]^d \times [0,1]^d} \mathbb{E} \big[ \|X(t) - X(s) \|^{-\gamma} \big]\,d\llambda_d(t)\,d\llambda_d(s) < \infty.
\end{equation}

Let $0<\lambda_1<\cdots< \lambda_p$ be the distinct real parts of the eigenvalues of $D$ with multiplicities $m_1,\ldots, m_p$. We will use the spectral decomposition with respect to $D$ as laid out in \cite{MS}. According to this, in an appropriate basis of $\rmm$, we can decompose $\rmm=V_1\oplus\cdots\oplus V_{p}$ into mutually orthogonal subspaces $V_i$ of dimension $m_i$ such that each $V_i$ is $D_i$-invariant, where the real part of any eigenvalue of $D_i$ is equal to $\lambda_i$ and $D=D_1\oplus\cdots\oplus D_{p}$ is block-diagonal. With respect to this spectral decompositon we may write $x\in\rmm$ as $x=x_1+\cdots+x_p$ with $x_i\in V_i\simeq\rr^{m_i}$ and we have $\|x\|^2=\|x_1\|^2+\cdots+\|x_p\|^2$ in the associated euclidean norms.

Now let $A=[-1,1]^{d+m}\setminus c^{E\oplus D}([-1,1]^{d+m})$ then $\bigcup_{j=0}^\infty c^{j(E\oplus D)}(A)=[-1,1]^{d+m}\setminus\{0\}$ is a disjoint covering. Using a change of variables $(x,t)=(c^{jD}y,c^{jE}s)$ together with the self-affinity of the random field and the b.c.i.\ condition, we get for some unspecified constant $K>0$ and every $\ell\in\{1,\ldots,p\}$
\begin{align*}
& \int_{[0,1]^d \times [0,1]^d} \mathbb{E} \big[ \|X(t) - X(s) \|^{-\gamma} \big]\,d\llambda_d(t)\,d\llambda_d(s)\\
& \quad\leq\int_{[-1,1]^d} \mathbb{E} \big[ \|X(t)\|^{-\gamma} \big]\,d\llambda_d(t)\\
& \quad\leq\int_{[-1,1]^d}\int_{[-1,1]^m}  \|x\|^{-\gamma}\,dP_{X(t)}(x)\,d\llambda_d(t)+2^d\\
& \quad=\sum_{j=0}^\infty\int_{c^{j(E\oplus D)}(A)}  \|x\|^{-\gamma}\,dP_{X(t)}(x)\,d\llambda_d(t)+2^d\\
&\quad=\sum_{j=0}^\infty c^{jq}\int_{A} \|c^{jD}y\|^{-\gamma}\,dP_{X(s)}(y)\,d\llambda_d(s)+2^d\\
&\quad=\sum_{j=0}^\infty c^{jq}\int_{A}\|c^{jD}y\|^{-\gamma} \,d\Exp[\tau_X](s,y)+2^d\\
&\quad\leq K\sum_{j=0}^\infty c^{jq}\int_{A}\|c^{jD}y\|^{-\gamma} \,d\llambda_{d+m}(s,y)+2^d\\
&\quad\leq K\sum_{j=0}^\infty c^{jq}\int_{[-1,1]^{m_1}}\cdots\int_{[-1,1]^{m_{\ell}}} \frac{d\llambda_{m_1}(x_1)\cdots d\llambda_{m_{\ell}}(x_{\ell})}{\big(\sum_{i=1}^{\ell}\|c^{jD_i}x_i\|\big)^{\gamma}}+2^d.
\end{align*}
By Theorem 2.2.4 in \cite{MS}, for any $\varepsilon>0$ there exists $K_1>0$ such that for $j\in\nat_0$ and every $i=1,\ldots,\ell$ we have $\|c	^{jD_i}x_i\|\geq\|c^{-jD_i}\|^{-1}\|x_i\|\geq K_1c^{j(\lambda_i+\varepsilon)}\|x_i\|$ and hence by change of variables $y_i=c^{j(\lambda_i-\lambda_{\ell})}x_i$ we get
\begin{align*}
& \int_{[-1,1]^{m_1}}\cdots\int_{[-1,1]^{m_{\ell}}} \frac{d\llambda_{m_1}(x_1)\cdots d\llambda_{m_{\ell}}(x_{\ell})}{\big(\sum_{i=1}^{\ell}\|c^{jD_i}x_i\|\big)^{\gamma}}\\
&\quad\leq K\int_{[-1,1]^{m_1}}\cdots\int_{[-1,1]^{m_{\ell}}} \frac{d\llambda_{m_1}(x_1)\cdots d\llambda_{m_{\ell}}(x_{\ell})}{\big(\sum_{i=1}^{\ell}c^{j(\lambda_i+\varepsilon)}\|x_i\|\big)^{\gamma}}\\
&\quad\leq K\,c^{-j\gamma(\lambda_{\ell}+\varepsilon)}\int_{[-1,1]^{m_1}}\cdots\int_{[-1,1]^{m_{\ell}}} \frac{d\llambda_{m_1}(x_1)\cdots d\llambda_{m_{\ell}}(x_{\ell})}{\big(\sum_{i=1}^{\ell}c^{j(\lambda_i-\lambda_{\ell})}\|x_i\|\big)^{\gamma}}\\
&\quad\leq K\,c^{-j\big(\gamma(\lambda_{\ell}+\varepsilon)+\sum_{i=1}^{\ell}(\lambda_i-\lambda_{\ell})m_i\big)}\int_{[-1,1]^{\widetilde m}}\|y\|^{-\gamma} \,d\llambda_{\widetilde m}(y),
\end{align*}
where $\widetilde m=m_1+\ldots+m_{\ell}$ and $y=y_1+\cdots+y_{\ell}$ with respect to the spectral decomposition of $\widetilde D=D_1\oplus\cdots\oplus D_{\ell}$. We now distinguish between the two cases considered in Example \ref{svalrange}.

{\bf Case 1:} Assume that $\sum_{i=1}^{\ell-1} \lambda_i m_i<q\leq\sum_{i=1}^{\ell} \lambda_i m_i$. For sufficiently small $\varepsilon>0$ we have $\gamma(\lambda_{\ell}+\varepsilon)<s(c^D,c^q)\lambda_{\ell}$. It suffices to consider large values of $\gamma<s(c^D,c^q)$ so that combining \eqref{c1condref} and \eqref{specialsc1} we may assume
$$\sum_{i=1}^{\ell-1} m_i+R-1<\gamma\leq\sum_{i=1}^{\ell-1} m_i+R$$
for some $R\in\{1,\ldots, m_{\ell}\}$. Hence for the singular value function by \eqref{svfct} and \eqref{jsev} we have for sufficiently large $j\in\nat$
\begin{align*}
\phi_{c^D}(\gamma) & \geq c^{j\sum_{i=1}^{\ell-1} (\lambda_i+\varepsilon)m_i}c^{j(\lambda_{\ell}+\varepsilon)(R-1)}c^{j(\lambda_{\ell}+\varepsilon)(\gamma-\sum_{i=1}^{\ell-1} m_i-r+1)}\\
& =c^{j\big((\lambda_{\ell}+\varepsilon)\gamma+\sum_{i=1}^{\ell-1} (\lambda_i-\lambda_{\ell})m_i\big)}.
\end{align*}
Note that by \eqref{ubsv} we have $\gamma<s(c^D,c^q)\leq\widetilde m$ and thus $\int_{[-1,1]^{\widetilde m}}\|y\|^{-\gamma} \,d\llambda_{\widetilde m}(y)<\infty$. By Lemma \ref{23svf} we further get as $j\to\infty$
$$c^{jq}c^{-j\big(\gamma(\lambda_{\ell}+\varepsilon)+\sum_{i=1}^{\ell} (\lambda_i-\lambda_{\ell})m_i\big)}\leq c^{jq}\phi^{-1}_{c^D}(\gamma)\to0,$$
which shows that 
$$\sum_{j=0}^\infty c^{jq}c^{-j\big(\gamma(\lambda_{\ell}+\varepsilon)+\sum_{i=1}^{\ell} (\lambda_i-\lambda_{\ell})m_i\big)}\int_{[-1,1]^{\widetilde m}}\|y\|^{-\gamma} \,d\llambda_{\widetilde m}(y)<\infty.$$

{\bf Case 2:} Assume that $\sum_{i=1}^p \lambda_i m_i<q$ then $s(c^D,c^q)=m$ and we choose $\ell=p$. For sufficiently small $\varepsilon>0$ we have $\gamma(\lambda_{p}+\varepsilon)<s(c^D,c^q)\lambda_{p}$. It suffices to consider large values of $\gamma<s(c^D,c^q)$ so that we may assume
$$\sum_{i=1}^{p-1} m_i+m_p-1<\gamma\leq\sum_{i=1}^{p-1} m_i+m_p=m=s(c^D,c^q).$$
Hence for the singular value function by \eqref{svfct} and \eqref{jsev} we have
\begin{align*}
\phi_{c^D}(\gamma) & \geq c^{j\sum_{i=1}^{p-1} (\lambda_i+\varepsilon)m_i}c^{j(\lambda_{p}+\varepsilon)(m_p-1)}c^{j(\lambda_{p}+\varepsilon)(\gamma-\sum_{i=1}^{p-1}m_i-m_p+1)}\\
& =c^{j\big((\lambda_{p}+\varepsilon)\gamma+\sum_{i=1}^{p-1} (\lambda_i-\lambda_{p})m_i\big)}
\end{align*}
for sufficiently large $j\in\nat$. Note that for $\ell=p$ we have $\gamma<s(c^D,c^q)=m=\widetilde m$ and thus $\int_{[-1,1]^{m}}\|y\|^{-\gamma} \,d\llambda_{m}(y)<\infty$. By Lemma \ref{23svf} we further get as $j\to\infty$
$$c^{jq}c^{-j\big(\gamma(\lambda_{p}+\varepsilon)+\sum_{i=1}^{p} (\lambda_i-\lambda_{p})m_i\big)}\leq c^{jq}\phi^{-1}_{c^D}(\gamma)\to0,$$
which shows that 
$$\sum_{j=0}^\infty c^{jq}c^{-j\big(\gamma(\lambda_{p}+\varepsilon)+\sum_{i=1}^{p} (\lambda_i-\lambda_{p})m_i\big)}\int_{[-1,1]^{m}}\|y\|^{-\gamma} \,d\llambda_{m}(y)<\infty.$$

Putting things together, we get \eqref{suff} in both cases, concluding the proof. 
\end{proof}

In a special situation we are able to get an analogue of Theorem \ref{31main} for the range.

\begin{cor}\label{eqrange}
If in addition to the assumptions of Theorems \ref{31main} and \ref{lbrange} we have $\lambda_p\leq a_1$, where $\lambda_p$ and $a_1$ are as in Example \ref{svalrange}, then with probability one
$$\dim_{\mathcal{H}}X ([0,1]^d)=s(c^D,c^q).$$
\end{cor}

\begin{proof}
By Theorem \ref{lbrange} $s(c^D,c^q)$ is a lower bound for $\dim_{\mathcal{H}}X ([0,1]^d)$ almost surely. If $s(c^D,c^q)=m$ in \eqref{srange} there is nothing to prove. Otherwise, if $\sum_{i=1}^{\ell-1} \lambda_i m_i<q\leq\sum_{i=1}^{\ell} \lambda_i m_i$, a comparison of \eqref{srange} with \eqref{specials} together with the assumption $\lambda_p\leq a_1$ directly shows that $s(c^{E\oplus D},c^q)=s(c^D,c^q)$. Thus by Theorems \ref{31main} and \ref{24cardim} we get the upper bound
$$\dim_{\mathcal{H}}X ([0,1]^d)\leq\dim_{\mathcal{H}}\Gr X ([0,1]^d)=s(c^{E\oplus D},c^q)=s(c^D,c^q)$$
almost surely, since we assumed \eqref{3Frostmancondition} and the b.c.i.\ condition.
\end{proof}


\section{Examples}

To demonstrate the applicability of our main results, we give examples of large classes of self-affine random fields for which \eqref{3Frostmancondition} holds and the precise values of the Hausdorff dimension of the graph and the range are already known.

\subsection{Operator-self-similar stable random fields} 

Let $E \in \mathbb{R}^{d \times d}$ and $D \in \mathbb{R}^{m \times m}$ be matrices and assume that the eigenvalues of $E$ and $D$ have {positive real part}. A random field $\{ X(t)\}_{t \in \rd }$ with values in $\rmm$ is said to be $(E,D)$-operator-self-similar if
\begin{equation*}
	\{ X( c^E t)\}_{t \in \mathbb{R}^d } \stackrel{\rm f.d.}{=} \{ c^D X(t)\}_{ t \in \mathbb{R}^d }  \quad \text{ for all } c>0.
\end{equation*}
These fields have been introduced in \cite{LiXiao} as a generalization of both operator scaling random fields \cite{BMS} and operator-self-similar processes \cite{HudsonMason,LahaRo}. Moreover, for $d=m=1$ one obtains the well-known class of self-similar processes.
We say that a random field $\{ X(t) : t \in \rd \}$ is {symmetric $\alpha$-stable (S$\alpha$S) for some} $\alpha \in (0,2]$ if any linear combination $\sum_{k=1}^m b_k X(t_k)$ is a symmteric $\alpha$-stable random vector. In \cite[Theorem 2.6]{LiXiao} it is shown that a proper, stochastically continuous $(E,D)$-operator-self-similar S$\alpha$S random field $X$ with stationary increments can be given by a harmonizable representation, provided that $0<\lambda_1\leq\ldots\leq\lambda_m<1$ for the real parts of the eigenvalues of $D$ and $1<a_1<\ldots<a_p$ for the distinct real parts of the eigenvalues of $E$.
This includes the case of operator fractional Brownian motions studied in \cite{MasonXiao,DP1,DP2} and operator scaling stable random fields \cite{BMS}, where corresponding Hausdorff dimension results already were already established in \cite{MasonXiao,BMS,BL}.
Note that, since the real parts of the eigenvalues of $E$ and $D$ are assumed to be positive, the matrices $c^E$ and $c^D$ are contracting for any $0<c<1$. In particular $X$ is a $(c^E, c^D)$-self-affine random field for any $0<c<1$, since its sample functions are continuous.

We now argue that the occupation measure $\tau_{X}$ satisfies the b.c.i.\ condition with respect to $W = c^E \oplus c^D=c^{E\oplus D}$. Recall that any symmetric $\alpha$-stable random variable has a smooth and bounded probability density (see \cite{SamorodTaqq}) so that the density $y\mapsto p_t(y)$ of $X(t)$ exists for all $t \neq 0$ and the mapping $(t,y)\mapsto p_t(y)$ is continuous due to stochastic continuity of the field $X$. By Lemma \ref{specialocm} we only have to prove that there is a constant $K>0$, not depending on $t$ and $y$, such that
$ p_t (y) \leq K$
for any $(t,y) \in [-1,1]^{d+m} \setminus W \cdot [-1,1]^{d+m}$. In order to show this, we will use generalized polar coordinates with respect to $E$, initially introduced in \cite{BMS}. For any $t \in \rd \setminus \{0\}$ one can uniquely write
$ t = \rho_E( t)^E l_E(t)$
with $E$-homogeneous radius $\rho_E(t) >0$ and direction vector $l_E(t) \in S_E = \{ t \in \rd : \rho_E(t) = 1 \}$. Note that $S_E$ is compact and does not contain $0$. The operator self-similarity implies
$$p_t (y) = ( \det c^D )^{-1} p_{c^{-E} t} (c^{-D} y) \quad\text{ for any } c>0,\,t\in\rd\setminus\{0\},\,y\in\rmm$$
{and} for $(t,y) \in [0,1]^{d+m} \setminus W\cdot [0,1]^{d+m}$ we get
\begin{align*}
p_t (y) & = p_{\rho_E( t)^E l_E(t) } (y) = \big( \det \rho_E(t) \big)^{-D} p_{\rho_E(t)^{-E} \rho_E(t)^E l_E(t) } \big( \rho_E(t)^{-D} y \big) \\
& = \rho_E(t) ^{-\trace (D)} p_{l_E(t) } \big( \rho_E(t)^{-D} y \big) \\
& \leq K_1 \cdot \max_{ \theta \in S_E} \sup_{y \in \rmm} p_\theta (y) \leq K,
\end{align*}
where $K_1,K>0$ are constants independent of $t$ and $y$. Hence, the b.c.i.\ condition holds and Theorem \ref{24cardim} allows to compute the carrying dimension of $\tau_{X}$ as
$$ \cardim \tau_{X} = s (W, {c^q} ) \quad\text{ almost surely for any }0<c<1,$$
where $W = c^{E \oplus D}$ and $q = \trace (E)$.

The Hausdorff dimension of the graph of $X$ has been computed in \cite[Theorem 4.1]{Soenmez1} for $\alpha=2$ and \cite[Theorem 5.1]{Soenmez2} for $\alpha \in (0,2)$, where the lower bound in the computation is proven through \eqref{3Frostmancondition}. Indeed it is shown that with probability one $\dim_{\mathcal{H}} \Gr X ([0,1]^d)$ coincides with
\begin{equation}\label{Hdimossrf}\begin{cases} 
\displaystyle\lambda_\ell^{-1}\Big(\sum_{k=1}^p a_k \mu_k + \sum_{i=1}^\ell (\lambda_\ell - \lambda_i)\Big) & \text{ if } \displaystyle\sum_{i=1}^{\ell-1} \lambda_i < \displaystyle\sum_{k=1}^p a_k \mu_k \leq \displaystyle\sum_{i=1}^{\ell} \lambda_i  ,\\ 		
\displaystyle\sum_{j=1}^\ell \frac{\tilde{a}_j}{\tilde{a}_\ell}  \tilde{\mu}_j + \displaystyle\sum_{j=\ell+1}^p  \tilde{\mu}_j +\displaystyle \sum_{i=1}^m \Big(1- \displaystyle\frac{\lambda_i}{\tilde{a}_\ell}\Big) & \text{ if } \displaystyle\sum_{k=1}^{\ell-1} \tilde{a}_k  \tilde{\mu}_k \leq\displaystyle \sum_{i=1}^m \lambda_i <\displaystyle \sum_{k=1}^{\ell} \tilde{a}_k  \tilde{\mu}_k ,
\end{cases}\end{equation}
where $\mu_1, \ldots, \mu_p$ denote the multiplicities of $a_1, \ldots, a_p$ respectively, $\tilde{a}_j = a_{p+1-j}$ and $\tilde{\mu}_j = \mu_{p+1-j}$ for $1 \leq j \leq p$. Since the assumptions of Theorem \ref{31main} are fulfilled, Corollary \ref{33computation} allows us to state that \eqref{Hdimossrf} coincides with $s(c^{E \oplus D}, {c^q} )$ for any $c\in(0,1)$, which can also be verified by elementary calculations using Example \ref{sexpl} as follows.\\
If $\sum_{i=1}^{\ell-1} \lambda_i <\sum_{k=1}^p a_k \mu_k \leq \sum_{i=1}^{\ell} \lambda_i$ then $r=\ell$ in \eqref{defr} and by \eqref{specials} we get
\begin{align*}
s(c^{E \oplus D}, {c^q} ) & =\ell-1+\frac1{\lambda_\ell}\Big(q-\sum_{i=1}^{\ell-1}\lambda_i\Big)\\
& =\lambda_\ell^{-1}\Big(\sum_{k=1}^p a_k\mu_k+\sum_{i=1}^{\ell}(\lambda_\ell-\lambda_i)\Big).
\end{align*}
On the other hand, if $\sum_{k=1}^{\ell-1} \tilde{a}_k  \tilde{\mu}_k \leq\sum_{i=1}^m \lambda_i <\sum_{k=1}^{\ell} \tilde{a}_k  \tilde{\mu}_k $, or equivalently
$$\sum_{i=1}^m\lambda_i+\sum_{k=1}^{p-\ell}a_k\mu_k<q\leq \sum_{i=1}^m\lambda_i+\sum_{k=1}^{p-\ell+1}a_k\mu_k$$
then we know that
$$r=m+\sum_{k=1}^{p-\ell}\mu_k+R\quad\text{ for some }R\in\{1,\ldots,\mu_{p-\ell+1}\}$$
in \eqref{defr} and by \eqref{specials} we get
\begin{align*}
s(c^{E \oplus D}, {c^q} ) & =m+\sum_{k=1}^{p-\ell}\mu_k+R-1\\
& \,\qquad+\frac1{a_{p-\ell+1}}\Big(q-\sum_{i=1}^m\lambda_i-\sum_{k=1}^{p-\ell}a_k\mu_k-a_{p-\ell+1}(R-1)\Big)\\
& =\sum_{j=\ell+1}^p  \tilde{\mu}_j+\frac1{a_{p-\ell+1}}\sum_{k=p-\ell+1}^p a_k\mu_k+\sum_{i=1}^m\Big(1-\frac{\lambda_i}{a_{p-\ell+1}}\Big)\\
& =\displaystyle\sum_{j=1}^\ell \frac{\tilde{a}_j}{\tilde{a}_\ell}  \tilde{\mu}_j + \displaystyle\sum_{j=\ell+1}^p  \tilde{\mu}_j +\displaystyle \sum_{i=1}^m \Big(1- \displaystyle\frac{\lambda_i}{\tilde{a}_\ell}\Big).
\end{align*}
Further, by \cite[Theorem 4.1]{Soenmez1} for $\alpha=2$ and \cite[Theorem 5.1]{Soenmez2} for $\alpha\in(0,2)$ we have almost surely
$$\dim_{\mathcal{H}} X ([0,1]^d)=\begin{cases}
m & \text { if } \displaystyle\sum_{i=1}^m\lambda_i<q,\\
\lambda_\ell^{-1}\Big(q+\displaystyle\sum_{i=1}^\ell (\lambda_\ell-\lambda_i)\Big) & \text { if } \displaystyle\sum_{i=1}^{\ell-1}\lambda_i<q\leq\displaystyle\sum_{i=1}^{\ell}\lambda_i.
\end{cases}$$
In accordance with Corollary \ref{eqrange}, a comparison with \eqref{srange} directly shows that this value coincides with $s(c^D,c^q)$ for any $c\in(0,1)$, indicating that the lower bound in Theorem \ref{lbrange} is in fact equal to the Hausdorff dimension of the range for the harmonizable representation of any $(E,D)$-operator-self-similar stable random field. All the above results also hold for a moving average representation of the random field in the Gaussian case $\alpha=2$ as shown in \cite{Soenmez1}. However, for a corresponding moving average representation in the stable case $\alpha\in(0,2)$, constructed in \cite{LiXiao}, it is questionable if our results are applicable, since these fields do not share the same H\"older continuity properties and thus the joint measurability of sample functions (assumption (iv) in the Introduction) may be violated; cf.\ \cite{BL,BL3} for details.

\subsection{Operator semistable L\'evy processes} 

To give an example of random fields that are not operator self-similar but possess the weaker discrete scaling property of self-affinity, we will now consider operator semistable L\'evy processes for $d=1$ with the restriction to $t\geq0$. 
Let $X=\{X(t)\}_{t\geq0}$ be a strictly operator-semi-selfsimilar process in $\rr^m$, i.e.
\begin{equation}\label{osslevy}
\{X(ct)\}_{t\geq0}\stackrel{\rm f.d.}{=}\{c^DX(t)\}_{t\geq0}\quad\text{ for some }c\in(0,1),
\end{equation}
where $D\in\rr^{m\times m}$ is a scaling matrix. If $X$ is a proper L\'evy process, it is called an operator semistable process and it is known that the real part of any eigenvalue of $D$ belongs to $[\frac12,\infty)$, where $\frac12$ refers to a Brownian motion component; see \cite{MS} for details. Hence this process can be considered as a self-affine random field with $d=1$ and  non-singular contractions $U=c$, and $V=c^D$. This includes operator stable L\'evy processes, where \eqref{osslevy} holds for any $c>0$, and multivariate stable L\'evy processes, where additionally $D$ is diagonal.  For these particular cases, Hausdorff dimension results for the range and the graph have been established in \cite{BKMS,MX,BG,PT,XLin,Hou}. Let $\frac12\leq \lambda_1<\ldots<\lambda_p$ denote the distinct real parts of the eigenvalues of $E$ with multiplicity $m_1,\ldots,m_p$, then recently Wedrich \cite{Wedrich} (cf.\ also \cite{KerMeeXia}) has shown that for any operator semistable L\'evy process $X$ almost surely
\begin{equation}\label{Lina}
\dim_{\mathcal H}\Gr X([0,1])=\begin{cases}
\max\{\lambda_1^{-1},1\} & \text{ if }\lambda_1^{-1}\leq m_1,\\
1+\max\{\lambda_2^{-1},1\}(1-\lambda_1) & \text{ else,}
\end{cases}
\end{equation}
where the lower bound in the computation is proven through \eqref{3Frostmancondition}. Moreover, in view of Lemma \ref{specialocm} it follows directly from \cite[Lemma 2.2]{Kern} that $\tau_X$ satisfies the b.c.i.\ condition. Since the assumptions of Theorem \ref{31main} are fulfilled, Corollary \ref{33computation} allows us to state that \eqref{Lina} coincides with $s(c\oplus c^D, c)=\cardim\tau_X$ irrespectively of $c\in(0,1)$, which can also easily be verified by elementary calculations. 
Further, by Corollary 3.2 and Theorem 3.3 in \cite{Kern} we have 
almost surely
\begin{equation}\label{Linarange}
\dim_{\mathcal H}X([0,1])=\begin{cases}
\lambda_1^{-1} & \text{ if }\lambda_1^{-1}\leq m_1,\\
1+\lambda_2^{-1}(1-\lambda_1) & \text{ if }\lambda_1^{-1}>m_1=1,\, m\geq2,\\
1 & \text{ if }\lambda_1^{-1}>m_1=1,\, m=1.
\end{cases}
\end{equation}
This value coincides with $s(c^D,c)$ as will be shown below, indicating that the lower bound in Theorem \ref{lbrange} is in fact equal to the Hausdorff dimension of the range for any operator semistable L\'evy process. Since \eqref{Linarange} shows that $\dim_{\mathcal H}X([0,1])\in(0,2]$ almost surely, it suffices to consider the singular value function $\phi_{W^n}(s)$ of $W=c^D$ for $s\in(0,2]$. We distinguish between the following cases.

{\it Case 1}: $\lambda_1\geq 1$, then for $s\in(0,1]$ the singular value function $\phi_{W^n}(s)$ asymptotically behaves as $c^{n\lambda_1s}$ in the sense of \eqref{jsev} showing that $s(c^D,c)=\lambda_1^{-1}$.

{\it Case 2}: $\lambda_1<1$ and $m=1$, then as in the first case for $s\in(0,1]$ the singular value function $\phi_{W^n}(s)$ asymptotically behaves as $c^{n\lambda_1s}$ in the sense of \eqref{jsev}. Due to the restriction $s\leq m=1$ we have $s(c^D,c)=1$.

{\it Case 3}: $\lambda_1<1$, $m_1=1$ and $m\geq2$, then for $s\in(1,2]$ the singular value function $\phi_{W^n}(s)$ asymptotically behaves as $c^{n(\lambda_1+\lambda_2(s-1))}$ in the sense of \eqref{jsev} showing that $s(c^D,c)=1+\lambda_2^{-1}(1-\lambda_1)$.

\vspace*{2ex}
The operator semistable L\'evy processes may be generalized to multiparameter operator semistable processes with $d\geq2$ as in \cite{Ehm,XLin} or to certain operator semi-selfsimilar strong Markov processes as in \cite{LiuX}, for which corresponding Hausdorff dimension results of the sample functions are not yet available in full generality from the literature. Our approach will give promising candidates for the Hausdorff dimension of the range and the graph of such fields in terms of the real parts of the eigenvalues of the scaling exponent. These serve at least as lower bounds by Theorems \ref{31main} and \ref{lbrange}, while corresponding upper bounds should be pursued elsewhere. We conjecture that these candidates are the precise Hausdorff dimension values.

\bibliographystyle{plain}

\end{document}